\documentclass[10pt]{article}
\title{A direct approach to the rational normal form}
\author{Klaus Bongartz}
\usepackage{german}\usepackage{amsthm} 
\usepackage{german}\usepackage{german}
\usepackage[utf8]{inputenc}
\newtheorem{definition}{Definition}
\newtheorem{theorem}{Theorem}
\newtheorem{lemma}{Lemma}

\begin{document}
\maketitle
In courses on linear algebra the rational normal form of a matrix  is usually derived from the structure theorem on  finitely generated  modules over $k[X]$ and/or from  operations on the rows and
 columns of the characteristic matrix. Here we propose a more direct approach
 staying in the realm of vector spaces and working with modules only implicitely. This 
should be  easier to understand for beginners in mathematics and it gives for $n=2,3,4$ a fast way to find the rational normal form of a matrix.

Throughout this note $k$ is an arbitrary  field. The set of all $n\times n$-matrices with coefficients in $k$ is denoted by $k^{n\times n}$. All irreducible polynomials as well as the
greatest common divisor gcd or the least common multiple lcm of two non-zero polynomials are normed.

\begin{definition}A rational normal form $R \in  k^{n\times n}$ is a bloc-diagonal matrix \[
\left [
\begin{array}{cccc}
B(P_{1})&0& ...&0\\
0&B(P_{2})& ... & 0\\
...&...&...&...\\
0&0&...&B(P_{r})
\end{array}
\right ]
\] with companion-matrices $B(P_{1}),B(P_{2}),\ldots ,B(P_{r})$ on the main diagonal, such that all polynomials $P_{i}$ are normed with coefficients in $k$ and such that  $P_{i+1}$ divides 
 $P_i$ for all $i$ with  $1\leq i \leq r-1$. We write $R=R(P_{1},P_{2},\ldots ,P_{r})$ for such a matrix and call it an RNF.
\end{definition}

We want to show:

\begin{theorem}Each matrix $A \in k^{n\times n}$ is similar over $k$ to exactly one rational normal form $R=R(P_{1},P_{2},\ldots ,P_{r})$. Here $R$ as well as an invertible matrix
 $T \in k^{n\times n}$ with $T^{-1}AT=R$ are obtained from the coefficients of $A$ by finitely many rational operations.
\end{theorem}

The proof uses only multiplication of matrices and polynomials, Gaussian elimination and the euclidian algorithm to determine the greatest common divisor of two polynomials. 
The  number of operations used to find $R$ and $T$ grows  polynomially with $n$.

We need  two well-known facts that are easy to prove.
\begin{lemma} \begin{enumerate}
\item For similar matrices $A$ and $B$ the characteristic polynomials $\chi_{A}$ and $\chi_{B}$ as well as the minimal polynomials  $\mu_{A}$ and $\mu_{B}$ coincide.
\item For any $H \in k[X]$ the matrices $H(A)$ and $H(B)$ are conjugate under $T$ provided $A$ and $B$ are so.
\item Let $M$ be a bloc-diagonal matrix with diagonal blocs $M_{i}$. Then the rank of $M$ is the sum of the ranks of the $M_{i}$'s, the characteristic polynomial $\chi_{M}$
is the product of the $\chi_{M_{i}}$ and for any $H \in k[X]$ the matrix $H(M)$ is a bloc-diagonal
 matrix with the $H(M_{i})$ as diagonal blocs.
              \end{enumerate}
 
\end{lemma}

\begin{lemma} For a companion-matrix $B(P)$ we have  $\chi_{B(P)}=\mu_{B(P)}=P$.
\end {lemma}

For the RNF $R=R(P_{1},P_{2},\ldots ,P_{r})$ of $A$ it follows that $\chi_{A}=\chi_{R}=\prod_{i=1}^{r}P_{i}$ and  $\mu_{A}=\mu_{R}=P_{1}$.
 Thus we get the refined Cayley-Hamilton-theorem saying 
that $\mu_{A}$ divides $\chi_{A}$ that in turn divides $\mu_{A}^{n}$, so that the characteristic polynomial and the minimal polynomial have the same irreducible factors in $k[X]$ 
and in particular the same roots.

We begin the proof of theorem 1 with the uniqueness of the RNF.

\begin{lemma} Two similar rational normal forms are equal.
\end{lemma}
 \begin{proof}Let $R=R(P_{1},\ldots ,P_{r})$ and $R'=R(Q_{1},Q_{2},\ldots,Q_{s})$ be RNF's with $T^{-1}RT=R'$ for  an invertible matrix  $T$
 with coefficients in $k$. We show $P_{i}=Q_{i}$  by induction and also $r=s$. 

Since $P_{1}$ is the minimal polynomial of $R$ and $Q_{1}$ is the minimal polynomial of $R'$  we get $P_{1}=Q_{1}$ from lemma 1. Suppose we know already $P_{j}=Q_{j}$ for all $j<i$. For 
 $i-1=r$ or $i-1=s$ we have  $r=s$ because $n$ equals the sum of the degrees of 
all $P_{j}$'s and $Q_{j}$'s. For $i-1<r$ we look at  $P_{i}(R)$ and $P_{i}(R')$ which have the same rank because they are similar. Only the first $i-1$ diagonal blocs of $P_{i}(R)$ are 
possibly non-zero and they coincide with the first $i-1$ blocs of $P_{i}(R')$ by the induction hypothesis. Since the ranks of $P_{i}(R)$ and $P_{i}(R')$ are equal we conclude that
 $P_{i}(B(Q_{i}))=0$. Therefore the minimal polynomial $Q_{i}$ of $B(Q_{i})$ divides $P_{i}$. By symmetry also $P_{i}$ divides $Q_{i}$ whence both are equal. 
  
 \end{proof}

The constructive proof of the existence is more difficult. We need the local minimal polynomial $\mu_{A,x}$ introduced in the next lemma. This  is just a normed generator of
 the annihilator of $x$ when we consider $k^{n}$ as a module
over $k[X]$ via $Xv=Av$.
\begin{lemma}
 For any $x\neq 0$ in $k^{n}$ there is a unique non-constant normed polynomial $\mu_{A,x}$ such that a polynomial $P$ satisfies $P(A)x=0$ iff it is a multiple of $\mu_{A,x}$.
\end{lemma}

\begin{proof}Let $m$ be the smallest natural number such that $-A^{m}x$ is a linear combination $a_{0}x + a_{1}Ax + \ldots a_{m-1}A^{m-1}x$. Then $X^{m} + a_{m-1}X^{m-1} + \ldots  + a_{0}$ is the 
wanted polynomial.
 
\end{proof}

 It is obvious that $\mu_{A}$ is the least common multiple of the $\mu_{A,x}$, but in fact the minimal polynomial is the local minimal polynomial of some appropriate vector that 
can be constructed in 
finitely many steps. The following simple fact about polynomials will be useful.

 \begin{lemma}
 Let $P$ and $Q$ be normed polynomials with lcm $V$ and gcd $G$ different from $P$ and $Q$. Suppose $P=G\tilde{P}$ and $Q=G\tilde{Q}$. Then one can construct a decomposition $G=HK$
 such that each prime divisor of $H$ occurs in $\tilde{Q}$
whereas $K$ and $\tilde{Q}$ are coprime. Then $V$ is the product of the two coprime polynomials $H\tilde{Q}$ and $K\tilde{P}$.
\end{lemma}
\begin{proof} Let $g$ be the degree of $G$ and let $H$ be the  gcd  of $G$ and $\tilde{Q}^{g}$. Then $G=HK$ is the wanted decomposition. Namely, let $R$ be an irreducible polynomial occurring  in $G$ with 
multiplicity $r>0$ and in $\tilde{Q}$ with non-zero multiplicity. Then $r\leq g$, whence $R^{r}$ occurs in $H$. The rest is easy to see.

\end{proof}
\begin{lemma} Let $A \in k^{n \times n}$ be given.
\begin{enumerate}
 \item For any non-zero $x$ and $y$ in $k^{n}$ we can construct a vector $z$ such that $\mu_{A,z}$ is the lcm $V$ of $\mu_{A,x}$ and $\mu_{A,y}$. 
\item One can construct a vector $x$ with $\mu_{A,x}=\mu_{A}$.
\end{enumerate}

\end{lemma}  \begin{proof}
 Let $G$ be  the greatest common divisor  of $P=\mu_{A,x}$ and $Q=\mu_{A,y}$. Then we have $P=G\tilde{P}$ and $Q=G\tilde{Q}$ and we can assume that $P\neq V \neq Q$.

If $G=1$ then $z=x+y$ has  $V=PQ$ as its local minimal polynomial.
 Namely, $V(A)$ annihilates $z$. Reversely, $L(A)z=0$ implies $L(A)x=-L(A)y$, whence $0=(PL)(A)x=-(PL)(A)y$. Thus $Q$ divides $PL$ and also $L$ because $P$ and $Q$ are coprime. By symmetry,
 $P$ divides $L$ too and so $PQ$ divides $L$.

For $G\neq 1$  we take  the decomposition $G=HK$ from the last lemma. Clearly, $\mu_{A,K(A)y}=H\tilde{Q}$ and $\mu_{A,H(A)x}=K\tilde{P}$. 
By the case treated before $V$ is the local minimal polynomial of  $z=H(A)x + K(A)y$.

The second assertion follows easily. One starts with an arbitrary vector $x\neq 0$ and computes $P=\mu_{A,x}$. If this differs from $\mu_{A}$ there is a canonical base vector $y=e_{i}$ 
outside of the kernel of $P(A)$. Using the first part one constructs a vector $z$ with $\mu_{A,z}$ of strictly larger degree. This procedure stops after at most $n$ steps with a wanted vector.

                                    \end{proof}

One should observe that the 'bad' vectors $x$ with $\mu_{A,x}\neq \mu_{A}$ belong to the union of the kernels of $S(A)$ where $S$ is one of the finitely many normed proper divisors of 
$\mu_{A}$. Thus over an infinite field like the real numbers 'most'  vectors $x$ satisfy $\mu_{A}=\mu_{A,x}$.

Now we prove the existence of the RNF $R$ and of a tranformation-matrix $T$ in three steps. Let $A =A_{0}\in k^{n\times n}$ be given. We will construct for $i=1,2,3$ three invertible matrices $T_{i}$ such that $A_{i}= T_{i}^{-1}A_{i-1}T_i$ is 
always a better 
'approximation' than $A_{i-1}$ to the wanted  RNF which is $A_{3}$ and $T=T_{1}T_{2}T_{3}$ is the transformation matrix.

In the first step we use lemma 6 to construct a vector $x$ with $P_{1}:=\mu_{A}=\mu_{A,x}$. In the following we denote by $p_{i}$ the degree of the polynomial $P_{i}$ occurring in the RNF to $A$.
We define the first matrix $T_{1}$ by taking as the first $p_{1}$ columns $x,Ax,\ldots ,A^{p_{1}-1}x$ and by completing this linearly independent set to a basis say by some appropriate
 canonical
 base vectors. The matrix $A_{1}=T_{1}^{-1}A_{0}T_1$ represents the original linear map given by $A$ with respect to the columns of $T_{1}$. Thus $A_{1}$ resp. $P_{1}(A_{1})$ are
 upper triangular bloc-matrices

\[\left [
\begin{array}{cc}
B(P_{1})&B_{2}\\
0&B_{4}
\end{array}
\right ]  resp. \left [
\begin{array}{cc}
P_{1}(B(P_{1}))&*\\
0&P_{1}(B_{4})
\end{array}
\right ].\]
We still have $P_{1}(A_{1})=0$, whence $P_{1}(B_{4})=0$. Thus the minimal polynomial $P_{2}$ of $B_{4}$ divides $P_{1}$. 

By induction there is an invertible matrix $T_{2}'$ with $n-p_1$ rows such that $T_{2}'^{-1}B_{4}T_{2}'$ is an RNF of the shape

\[\left [
\begin{array}{cccc}
B(P_{2})&0& ...&0\\
0&B(P_{3})& ... & 0\\
...&...&...&...\\
0&0&...&B(P_{r})
\end{array}
\right ].\]

Defining $T_{2}=\left [
\begin{array}{cc}
E_{p_1}&0\\
0&T_{2}'
\end{array}
\right ]$, where $E_{p_{1}}$ is the identity matrix with $p_{1}$ rows,   we obtain $A_{2}=T_{2}^{-1}A_{1}T_2$ of the shape \[\left [
\begin{array}{cccc}
B(P_{1})&C_{2}& ...&C_{r}\\
0&B(P_{2})& ... & 0\\
...&...&...&...\\
0&0&...&B(P_{r})
\end{array}
\right ].\] 

In the last step  the $C_{i}$'s in the first row will be made to $0$.
For each index $i\geq 2$ let  $v_{i}$ be the canonical base vector of $k^{n}$ with index $p_{1}+p_{2}+\ldots +p_{i-1}+1$ and define  $v_{1}=e_1$. Then $A_{2}^{j}v_{1}=e_{j+1}$ for $j < p_1$. 
Furthermore  $P_{i}(A_{2})v_{i}$ is the column of  $P_{i}(A_{2})$ with index $p_{1}+p_{2}+\ldots +p_{i-1}+1$. Only the first $p_1$ coefficients thereof might be non-zero,
whence $P_{i}(A_{2})v_{i}= H_{i}(A_{2})v_{1}$ for a uniquely determined polynomial $H_i$ of degree $p_{1}-1$ at most. Since $P_{2}$ and therefore all  $P_i$ divide $P_1$, 
there exist polynomials $R_i$ satisfying  $P_{1}= P_{i}R_{i}$. We infer $0=P_{1}(A_{2})v_{i}= (R_{i}P_{i})(A_{2})v_{i}=(R_{i}H_{i})(A_{2})v_{1}$, whence  
 $P_{1}=\mu_{A_{2},v_{1}}$ divides $R_{i}H_{i}$ and we obtain $H_{i}=P_{i}S_{i}$ for some $S_{i}$. Now consider $w_{i}= v_{i}- S_{i}(A_{2})v_1$ for $i=2,3,\ldots ,r$.
By construction we have $P_{i}(A_{2})w_{i}=0$. Moreover the matrix $T_3$  with columns
\[ v_{1},A_{2}v_{1}, \ldots, A_{2}^{p_{1}-1}v_{1},w_{2},A_{2}w_{2},\ldots ,A_{2}^{p_{2}-1}w_{2}, \ldots ,w_{r},A_{2}w_{r}, \ldots ,A_{2}^{p_{r}-1}w_{r}\]

is invertible being  upper triangular with $1$ as diagonal entries. $T_{3}^{-1}A_{2}T_{3}$ represents the linear map given by $A_{2}$ with respect to the columns of $T_{3}$ and so it is
 the wanted RNF.\vspace{0.5cm}

Note that for nilpotent $A$  the RNF  coincides with the Jordan normal form JNF. In that case step one and three of our proof are almost trivial and so we obtain in particular a short
proof for the  JNF in the nilpotent case to which the general case can be  reduced by looking at the generalized eigenspaces separately.

 For the convenience of the reader we provide another proof that we learned from Gabriel. This proof  does not proceed by induction and it gives a very clear picture
of the structure of nilpotent linear maps. 

\begin{theorem} Let $f$ be a nilpotent endomorphism of a finite dimensional vector space $V$. Then there is an ordered basis in which $f$ is represented by a nilpotent Jordan normal form.

\end{theorem}
\begin{proof}  Assume $f^{r+1}=0\neq f^r$. In the kernel $K$ of $f$ we look at the chain of subspaces 
\[K\cap f^{r}V \subseteq K\cap f^{r-1}V \subseteq \ldots K\cap fV \subseteq K.\] We start with a basis $x_{r,1},x_{r,2}, \ldots ,x_{r,n_{r}}$ of $K\cap f^{r}V$. Then we complete this 
by appropriate vectors  $x_{r-1,1},x_{r-1,2},\ldots ,x_{r-1,n_{r-1}}$ to a basis of $K\cap f^{r-1}V$ and so on. 
Here the case $n_{i}=0$ can occur for some indices. However  \[\mathcal{B'}=\{x_{i,j}|0\leq i\leq r,1\leq j \leq n_{i}\}\] is a basis of $K$. 
By construction there are vectors $y_{i,j}$ such that  $f^{i}y_{i,j}=x_{i,j}$ and we claim that \[\mathcal{B}=\{f^{l}y_{i,j}|0\leq i\leq r,1\leq j \leq n_{i}, 0\leq l \leq i\}\] 
is a basis of $V$. Before we prove this we illustrate the situation for $r=3$  by writing $\mathcal{B}$ into the following scheme similar to a staircase:
\vspace{0.1cm} 
\[\begin{array}{cccccccccccc}
y_{3,1}&.. &y_{3,n_{3}}&&&&&&&&&\\
fy_{3,1}&.. &fy_{3,n_{3}}&y_{2,1}&.. &y_{2,n_{2}}&&&&&&\\
f^{2}y_{3,1}&.. &f^{2}y_{3,n_{3}}&fy_{2,1}&..&fy_{2,n{2}}&y_{1,1}&..&y_{1,n_{1}}&&&\\
f^{3}y_{3,1}&..&f^{3}y_{3,n_{3}}&f^{2}y_{2,1}& ..&f^{2}y_{2,n_{2}}&fy_{1,1}&..&fy_{1,n_{1}}&y_{0,1}&..&y_{0n_{0}}\\
\end{array}\]\vspace{0.1cm}

Ordering the columns from the left to the right and the vectors in each column from the top to the bottom $f$ is represented by a JNF.

To see that $\mathcal{B}$ is linearly independent suppose $\sum \lambda_{l,i,j}f^{l}y_{i,j} = 0$. Applying $f^r$ to this equation we get $0=\sum \lambda_{l,i,j}f^{r+l}y_{i,j} = \sum \lambda_{0,r,j}x_{r,j},$ 
whence $\lambda_{0,r,j}=0$ for all  $j$. Next we apply  $f^{r-1}$ and we get  $0=\sum \lambda_{l,i,j}f^{l+r-1}y_{i,j} = \sum \lambda_{1,r,j}x_{r,j} + \sum \lambda_{0,r-1,j}x_{r-1,j} = 0$ 
 whence all  $\lambda_{1,r,j}$ and all  $\lambda_{0,r-1,j}$ are $0$. Continuing like that one finds that all coefficients $\lambda_{l,i,j}$ have to vanish.

To show that the span $W$ of $\mathcal{B}$ is $V$ we prove by induction on $i$ that $Ker f^i$ is contained in $W$ for $i=1,2,\ldots r+1$. This is clear for $i=1$.
In the inductive step from $i-1$ to $i$ take $v$ in $Ker f^i$. Then  $f^{i-1}v$ lies in $K\cap f^{i-1}V$, i.e. we have $f^{i-1}v= \sum_{k\geq i-1}\lambda_{kj}x_{kj}$ for some 
appropriate scalars $\lambda_{kj}$. Then $w=\sum_{k\geq i-1}\lambda_{kj}f^{k-i+1}y_{kj}$ lies in $W$ and we have $f^{i-1}v=f^{i-1}w$. Thus $v-w$ lies in $Ker f^{i-1}$ and $v=w+(v-w)$
 lies in $W$ by induction.
\end{proof}

We end this note with some remarks that are not all adressed to beginners.
\begin{itemize}
 \item The number $n_{i}$ occurring in the last proof is just the number of Jordan blocs of size $i+1$ and this implies the uniqueness of the JNF. 
\item The theorem of the RNF provides in particular a finite algorithm to decide whether two matrices $A$ and $B$ are similar or not. Furthermore the algorithm needs only rational operations
 with the entries of $A$ and $B$ and it is of polynomial complexity.  

The similarity problem for pairs of matrices is 'wild' and there is no method known to find a normal form using only 
rational operations. Belitzky's algorithm described in  \cite{Bel,Ser}   depends on the knowledge of the eigenvalues. Nevertheless in \cite{Bo2,Bo3}   we give  two different methods to
 decide  by rational
 operations whether two finite dimensional modules over any finitely generated algebra - e.g. two pairs of matrices -
 are isomorphic. The number of operations needed  grows only polynomially but very fast with the dimension, whence the algorithms are rather of theoretical interest.

\item For any partition $p=(p_{1},p_{2}, \ldots p_{r})$ of $n$ the set $S(p)$ of all matrices  $A$ having a RNF $R(A)=R(P_{1},P_{2},\ldots ,P_{r})$ such that the degree of $P_{i}$ is $p_{i}$ is
is a smooth rational locally closed $Gl_{n}$- invariant subvariety of $k^{n \times n}$ and the map $A\mapsto R(A)$ is a smooth morphism from $S(p)$ to affine $p_{1}$-space. All this 
is contained in \cite{Bo1} and it  is a  nice special case of the theory of sheets as described in \cite{Bor,Kra}.
\end{itemize}

\end{document}